%% file: article.tex
\documentclass[reqno]{amsart}
\usepackage{enumerate}
\input{macros_english}

\title[The dynamics of algorithmically random points]{Effective symbolic dynamics, random points, statistical behavior, complexity and entropy\footnote{partly supported by ANR Grant  05 2452 260 ox}}

\begin{document}

\author{Stefano Galatolo}
\address{Dipartimento di Matematica Applicata, Universita di Pisa}
\email{s.galatolo@docenti.ing.unipi.it}
\author{Mathieu Hoyrup}
\address{LIENS, Ecole Normale Sup\'erieure, Paris}
\email{hoyrup@di.ens.fr}
\author{Crist\'obal Rojas}
\address{LIENS, Ecole Normale Sup\'erieure and CREA, Ecole Polytechnique,
Paris}
\email{rojas@di.ens.fr}

\begin{abstract}
We consider the dynamical behavior of Martin-L\"of random points in
dynamical systems over metric spaces with a computable dynamics and
a computable invariant measure. We use computable partitions to
define a sort of effective symbolic model for the dynamics. Through
this construction we prove that such points have typical statistical
behavior (the behavior which is typical in the Birkhoff ergodic
theorem) and are recurrent. We introduce and compare some notions of
complexity for orbits in dynamical systems and prove: (i) that the
complexity of the orbits of random points equals the Kolmogorov-Sina\"i entropy of the
system, (ii) that the supremum of the complexity of orbits equals the
topological entropy.
\end{abstract}
\maketitle
%\textit{keywords:} Computable partitions, effective symbolic dynamics, Kolmogorov complexity, Martin-L\"of randomness.

\tableofcontents

\section{Introduction}

The  randomness of a particular outcome is always relative to some
statistical test. The notion of algorithmic randomness, defined by
Martin-L\"of in 1966, is an attempt to have an ``absolute'' notion of
randomness. This absoluteness is actually relative to all
``effective'' statistical tests, and lies on the hypothesis that this
class of tests is sufficiently wide.

Martin-L\"of's original definition was given for infinite symbolic
sequences. With this notion each single random sequence behaves as a
generic sequence of the probability space for each effective statistical
test. In this way many probabilistic theorems having almost everywhere
statements can be translated to statements which hold \emph{for each}
random sequence. As an example we cite the fact that in each infinite
string of 0's and 1's which is random for the uniform measure, all the
digits appear with the same limit frequency.
This is a particular case, related to the strong law of large numbers (or Birkhoff ergodic theorem). A general statement of this kind was given by V'yugin (Birkhoff ergodic theorem for individual random sequences, see \cite{Vyu97} and lemma \ref{vyugin} below).

Recently the notion of Martin-L\"of randomness was generalized to
computable metric spaces endowed with a measure (\cite{Gac05, HoyRoj07}). Computable metric spaces are separable metric spaces where the distance can be in some sense effectively computed
(see section \ref{CMS} ). In those spaces, it is also possible to define
``computable'' functions, which are functions whose behavior is in some
sense given by an algorithm, and ``computable'' measures (there is an
algorithm to calculate the measure of nice sets). The space of infinite
symbolic sequences, the real line or euclidean spaces, are examples of
metric spaces which become computable in a very natural way.

A particularly interesting class of general stationary stochastic processes
is constituted by those generated by a measure-preserving map on a metric
space, these are the objects studied by ergodic theory.  In this paper we
consider systems of the type $(X,T,\mu)$, where $X$ is a computable metric
space, $\mu$ a computable probability measure and $T$ a computable
endomorphism. The above considered symbolic shifts on spaces of infinite
sequences which preserve a computable measure are systems of this kind.

In the classical ergodic theory, a powerful technique (symbolic dynamics)
allows to associate to a general system as above $(X,T,\mu )$ a shift on a
space of infinite strings having similar statistical properties. In section
\ref{esd} we use the algorithmic features of computable metric spaces to
define computable measurable partitions and construct effective symbolic
models for the dynamics. In this models {\em random points are associated
to random infinite strings}.  This tool allows to generalize theorems
which are proved in the symbolic setting to the more general setting of
maps and metric spaces.  For example the above cited V'yugin theorem
becomes a Birkhoff theorem for random points. On this line, we also prove a
Poincar\'e recurrence theorem for random points. Those statements (see
thm.\ref{birkhoff_theorem} and prop. \ref{poinc}) can be summarized as
\begin{theorem_star}
Let $(X,\mu)$ be a computable probability space. If $x$ is $\mu$-random,
then it is recurrent with respect to {\em every } measure preserving
endomorphism $T$ on $(X,\mu)$.  Moreover, each $\mu$-random point $x$ is
typical for {\em every} ergodic endomorphism $T$, i.e.
\begin{equation}
\lim_{n \to \infty} \frac{1}{n} \sum_{j=0}^{n-1}f(T^jx) =\int f d\mu
\end{equation}
for every continuous bounded real-valued $f$.
\end{theorem_star}

In the remaining part of the paper these tools are also used to prove
relations between various definitions of orbit complexity and entropy of
the systems.

%The well known notion of measure theoretic entropy (see section
%\ref{section_measure_entropy}) of a dynamical system (also called Kolmogorov-Sinai
%entropy) comes directly from Shannon theory of information. The
%entropy of a system is a measure of the rate of Shannon information
%which is necessary to describe the dynamics. We remark that Shannon
%information is a global average notion, which depends on the
%probability measure which is considered on the space.

%In 1965, Kolmogorov defined an algorithmic notion of information content of
%a single string. This information does not depend on the measure and was
%actually intended to provide an absolute notion of information and
%individual randomness. In this setting a sequence is likey to be called
%random if it contains maximal information, in some sense.  Martin-L\"of
%proved that the formalization of this idea was not trivial and proposed a
%different definition. Later, the original idea of Kolmogorov was refined,
%and was proved to give the notion of Martin-L\"of randomness (see theorem
%\ref{theorem_deficiency}).

In \cite{Bru83}, Brudno defined a notion of algorithmic
complexity $\uK(x,T)$ for the orbits of
a dynamical system on a compact space. It is a measure of the information
rate which is necessary to describe the behavior of the orbit of $x$. In
this point-wise definition the information is measured by the Kolmogorov
information content. Later, White (\cite{Whi93}) also introduced a slightly different version
$\lK(x,T)$. Brudno then proved the following results, later improved by White:
\begin{theorem_star}[Brudno, White]\label{brudno_theorem}
Let $X$ be a compact topological space and $T:X\to X$ a continuous map.
\begin{enumerate}
\item For any ergodic probability measure $\mu$ the equality
$\lK(x,T)=\uK(x,T)=h_\mu(T)$ holds for $\mu$-almost all $x\in X$,
\item For all $x\in X$, $\uK(x,T)\leq h(T)$.
\end{enumerate}
\end{theorem_star}
where $h_\mu(T)$ is the Kolmogorov-Sina\"i entropy of $(X,T)$ with respect
to $\mu$ and $h(T)$ is the topological entropy of $(X,T)$. This result
seems miraculous as no computability assumption is required on the space or
on the transformation $T$. Actually, this miracle lies in the compactness
of the space, which makes it finite when observations are made with finite
precision (open covers of the space can be reduced to \emph{finite} open
covers). Indeed, when the space is not compact, it is possible to construct
systems for which the algorithmic complexity of orbits is correlated in no
way to their dynamical complexity. In \cite{Gal00}, Brudno's definition was
generalized to non-compact computable metric spaces. This definition
coincides with Brudno's one in the compact case and will be given in
section \ref{oc}.

The above definitions of orbit complexities follow a topological
approach. We show that the measure-theoretic setting also provides a
natural notion of orbit complexity $\Ksym(x,T)$ defined by computable
partitions.  This kind of orbit complexity will be defined almost
everywhere and in particular at each $\mu$-random point. For this notion
the first result in Brudno and White's theorem comes easily. We go further
in showing:
\begin{theorem_star}[\ref{final}]
Let $T$ be an ergodic endomorphism of the computable probability space $(X,\mu)$,
$$
\Ksym(x,T)=h_\mu(T) \quad\text{ for all $\mu$-random point $x$.}
$$
\end{theorem_star}

We then prove that the two notions of orbit complexity coincide on
Martin-L\"of random points:
\begin{theorem_star}[\ref{shadsym}]
Let $T$ be an ergodic endomorphism of the computable probability space
$(X,\mu)$, where $X$ is compact,
$$
\Ksym(x,T)=\uK(x,T)\quad\text{ for all $\mu$-random point $x$.}
$$
\end{theorem_star}

In the topological context, we then consider $\uK(x,T)$ and strengthen the
second part of Brudno's theorem, showing:
\begin{theorem_star}[\ref{theorem_top_entropy}]
Let $T$ be a computable map on a compact computable metric space $X$,
$$
\sup_{x\in X}\lK(x,T)=\sup_{x\in X} \uK(x,T)=h(T)
$$
\end{theorem_star}
Remark that this was already implied by Brudno's theorem, using the
variational principle: $h(T)=\sup\{h_\mu(T):\mu\text{ is
$T$-invariant}\}$. Nevertheless, our proof uses purely topological and
algorithmic arguments and no measure. In particular, it does not use the
variational principle, and can be thought as an alternative proof of it.

Many of these statements require that the dynamics and the invariant
measure are computable.

The first assumption can be easily checked on concrete systems if
the dynamics is given by a map which is effectively defined.

 The second is more delicate: it is well known that given a map on a metric
space, there can be a continuous (even infinite dimensional) space of
probability measures which are invariant for the map, and many of them will
be non computable. An important part of the theory of dynamical systems is
devoted to selecting measures which are particularly meaningful. From this
point of view, an important class of these measures is the class of SRB
invariant measures, which are measures being in some sense the ``physically
meaningful ones''(for a survey on this topic see \cite{Y}). It can be
proved (see \cite{GalHoyRoj07a} and \cite{GalHoyRoj07c} and their
references e.g.) that in several classes of dynamical systems where SRB
measures are proved to exist, these measures are also computable from our
formal point of view, hence providing several classes of nontrivial
concrete examples where our results can be applied.

%*********************************
%*********************************
\section{Preliminaries}\label{Prelim}

\subsection{Partial recursive functions}
\label{partial_rec}

The notion of algorithm working on integers has been formalized
independently by Markov, Church, Turing among others. Each
constructed model defines a set of partial (not defined everywhere)
integer functions which can be computed by some \emph{effective}
mechanical or algorithmic (w.r.t. the model) procedure. Later, it
has been proved that all this models define the same class of
functions, namely: the set of \emph{partial recursive functions}.
This fact supports a working hypothesis known as Church's Thesis,
which states that every (intuitively formalizable) algorithm is a
partial recursive function. It gives the connection between the
informal notion of algorithm and the formal definition of recursive
function.

%Usually one does not formally verify that some apparently recursive
%function is indeed
%recursive; instead one exhibits an algorithm which compute the
%function and Church's Thesis is invoked to guarantee that the function is
%in fact recursive. We shall do likewise.

Let us say then that a \defin{recursive function} is a function (on integers) that can be
computed in some \emph{effective} or \emph{algorithmic} way. For
formal definitions see for example, \cite{Rog87}. With this intuitive
description it is more or less clear that there exists an effective
procedure to enumerate the class of all partial recursive functions,
associating to each of them its \defin{G\"odel number}, which is the number of the program computing it. Hence there
exists a universal recursive function $\FI_u:\N \to \N$ satisfying
for all $e,n\in \N$,  $\FI_u(\uple{e,n})=\FI_e(n)$ where $e$ is the
g\"odel number of $\FI_e$ and $\uple{ , }:\N^2 \to \N$ is some
recursive bijection. In classical recursion theory, a set of natural
numbers is called \defin{recursively enumerable} (\defin{r.e} for
short) if it is the range of some partial recursive function. That
is if there exists an algorithm listing the set.  We denote by $E_e$
the r.e set associated to $\FI_e$. Namely
$E_e=rang(\FI_e)=\{\FI_u(\uple{e,n}):n\in \N\}$.

\subsection{Algorithms on finite objects}

Strictly speaking, algorithms only work on integers. However, when the
objects of some class have been identified with integers, it makes sense to
speak about algorithms acting on these objects.

\begin{definition}
A \defin{Numbered Set} $\O$ is a countable set together with a surjection
$\nu_\O:\N \to \O$ called the \defin{numbering}. We write $o_n$ for $\nu(n)$ and call $n$ the \defin{name}
of $o_n$.
\end{definition}

Of course, the potential of algorithms depends on the choice of the
numbering, since it determines to what extent an object can be
algorithmically recovered from its name. If the objects of some
collection can be characterized by a finite number of integers, then
the collection is a numbered set since its objects can be indexed
using standard recursive bijections from $\N^*$ to $\N$ and from
$\N^k$ to $\N$, which we will both denote $\uple{.}$.

\begin{examples}

\item $\Q$, with some standard numbering $\nu_{\Q}$ is a numbered set.

\item The set of partial recursive functions $\mathcal{R}=\{\FI_e:e\in \N\}$ is
a numbered set, G\"odel numbers being the names.

\item The collection $\{E_e=rang(\FI_e):e\in \N\}$ of all r.e subsets of $\N$ is a numbered set.

\end{examples}

\begin{definition}

Let $\O$ be a numbered set. To any recursive function $\FI:\N \to \N$ we associate
an \defin{algorithm} $\A_\FI:\N\to\O$ defined by
$\A_\FI(i)=o_{\FI(i)}$.
\end{definition}

Given a total algorithm $\A:\N\to\O$ (i.e $\A=\A_\FI$ for some total $\FI$), we say that the sequence of finite objects $(\A(n))_{n\in\N}$ is \defin{enumerated by} $\A$, or that $\A$ is an \defin{algorithmic enumeration} of this sequence.

\subsection{Computability over the reals}

\begin{definition}

Let $x$ be a real number. We say that:

\puce $x$ is \defin{lower semi-computable} if the set $E:=\{i\in \N :q_i<x\}$ is r.e,

\puce  $x$ is \defin{upper semi-computable} if the set $E:=\{i\in \N :q_i>x\}$ is r.e,

\puce $x$ is \defin{computable} if it is lower and upper semi-computable.

\end{definition}

Equivalently, a real number is computable if and only if there exists an algorithmic enumeration of a sequence of rational numbers converging exponentially fast to $x$. That is:

\begin{proposition} A real number $x$ is computable if and only if there exists an algorithm $\A:\N \to \Q$ such that   $|\A(i)-x|<2^{-i}$, for all $i$.
\end{proposition}

\begin{definition}Let $(x_n)_n$ be a sequence of computable reals. We say that the sequence is \defin{uniformly} computable or that $x_n$ is computable \defin{uniformly in n} if there exists an algorithm $\A:\N \to \Q$ such that for all $n$ and $i$ it holds $|\A(\uple{n,i})-x_n|<2^{-i}$.
\end{definition}

Uniform sequences of lower (upper) semi-computable reals are defined in the same way.

%*********************************************************
%*********************************************************
\subsection{Computable Metric Spaces}\label{CMS}

\begin{definition}A \defin{computable metric space} (CMS) is a triple $\X=(X,d,\S)$,
  where

\puce $(X,d)$ is a separable complete metric space.

\puce $\S=(s_i)_{i \in \N}$ is a numbered dense subset of $X$ (called  \defin{ideal points}).

\puce The real numbers $(d(s_i,s_j))_{\uple{i,j}\in\N}$ are all computable, uniformly in $\uple{i,j}$.

\end{definition}

We now recall some important examples of computable metric spaces:

\begin{examples}

\item the Cantor space $\mathcal(\Sigma^{\N},d,S)$ with $\Sigma$ a
finite alphabet and $d$ the usual distance. $S$ is the set of ultimately
$0$-stationary sequences.

\item $(\R^n,d_{\R^n},\Q^n)$ with the euclidean metric and the standard numbering of
$\Q^n$.

\item if $(X_1,d_1,S_1)$ and $(X_2,d_2,S_2)$ are two computable metric spaces, the distance $d((x_1,x_1),(y_1,y_2))=\max(d_1(x_1,y_1),d_2(x_2,y_2))$  on the product space $(X_1\times X_2,d,S_1\times S_1)$ makes it a computable metric space.
\end{examples}

For further examples we refer to \cite{Wei93}.

Let $(X,d,\S)$ be a computable metric space. The computable
structure of $X$ assures that the whole space can be ``reached''
using algorithmic means. Since ideal points (the finite objects of
$\S$) are dense, they can approximate any $x$ at any finite
precision. Then, $x$ itself can be identified to a sequence of ideal
points converging to $x$ in an effectively controlled way. Let us
say that a sequence of ideal points $(s_{i_n})_n$ is \defin{fast} if
$d(s_{i_n},s_{i_{n+1}})<2^{-n}$ for all $n$. As the space is
complete, a fast sequence has always a limit $x$, and
$d(s_{i_n},x)<2^{-(n-1)}$ for all $n$.

\begin{definition}[Computable points]
A point $x\in X$ is said to be \defin{computable} if there exists an algorithm $\A:\N \to \S$ which enumerates a fast sequence whose limit is $x$.
\end{definition}

As for real numbers we can give the notion of uniform sequence
\begin{definition}Let $(x_n)_n$ be a sequence of computable points. We say that the sequence is \defin{uniformly} computable or that $x_n$ is computable \defin{uniformly in n} if there exists an algorithm $\A:\N \to \S$ such that for all $n$, the sequence $(\A(\uple{n,i}))_i$ is fast and converges to $x_n$.
\end{definition}

The numbered set of ideal points $(s_i)_i$ induces the numbered set of \defin{ideal balls}
$\mathcal{B}:=\{B(s_i,q_j):s_i \in S, q_j \in\Q_{>0}\}$. We denote by $B_{\uple{i,j}}$ the ideal ball $B(s_i,q_j)$.

Computability can then be extended from the numbered set $\mathcal{B}$ to the space of
open subsets of $X$: such an open subset $U\subseteq X$ can be identified
to a collection of ideal balls whose union is $U$.

\begin{definition}[R.e open sets]We say that the set $U\subset X$ is \defin{r.e open}
if there is some r.e set $E\subset \N$ such that $U=\cup_{i\in E}B_i$.
\end{definition}

\begin{remark} Let $U$ be a r.e open set. It is easy to see that there is an algorithm to \defin{semi-decide} weather some ideal point belongs to $U$. That is, the algorithm will halt on input $i$ iff $s_i\in U$. This notion can be extended to any point $x$ in the following sense: The algorithm sequentially asks (from an external user) for finite approximations of $x$ at required precisions. If $x\in U$ the algorithm will eventually stop and answer ``yes'', if $x \notin U$ then the algorithm will run and ask forever. For formal definitions we refer to \cite{HoyRoj07}.
\end{remark}

\begin{definition}Let $(U_n)_n$ be a sequence of r.e open sets. We say that the sequence is \defin{uniformly r.e} or that $U_n$ is r.e open \defin{uniformly in n} if there exists an r.e set $E\subset \N$ such that for all $n$ it holds $U_n=\cup_{i\in E_n} B_{i}$, where $E_n=\{i:\uple{n,i}\in E\}$.
\end{definition}

\begin{examples}

\item If the sequence $(U_n)_n$ is uniformly r.e then the union $\cup_n U_n$ is a r.e open set.

\item The universal recursive function $\FI_u$ makes the collection  of all r.e open sets (denoted $\mathcal{U}$) a sequence uniformly r.e. Indeed, define $E:=\{\uple{e,\FI_u(\uple{e,n})}: e,n \in \N\}$. Then $\mathcal{U}=\{U_e:e\in N\}$ where $U_e=\cup_{i\in E_e}B_i$.

 \item The numbered set $\mathcal{U}$ is closed under finite unions and finite intersections. Furthermore, these operations are \emph{effective} in the following sense: there exists recursive functions $\FI^{\cup}$ and  $\FI^{\cap}$ such that for all $i,j\in \N$, $U_i\cup U_j=U_{\FI^{\cup}(\uple{i,j})}$ and the same holds for $\FI^{\cap}$. Equivalently: $U_i\cup U_j$ is r.e open uniformly in $\uple{i,j}$ . See \cite{HoyRoj07}.

\end{examples}

\begin{definition}[Constructive $G_{\delta}$-sets] We say that the set $D\subset X$ is a \defin{constructive $\boldsymbol{G_\delta}$-set} if it is the intersection of a sequence of uniformly r.e open sets.

\end{definition}

Let $(X,S_X,d_X)$ and $(Y,S_Y,d_Y)$ be computable metric spaces with $\mathcal{U}^X$ and $\mathcal{U}^Y$ the corresponding numbered sets of r.e open sets.

\begin{definition}[Computable Functions]\label{functions}
A function $T: X \to Y$ is said to be \defin{computable} if $T^{-1}(B_n)$ is r.e open uniformly in $n$.
\end{definition}

\begin{remark}
 We remark that this definition implies that the preimage of a uniform sequence of r.e. open sets is a uniform sequence of r.e. open sets. This could be an alternative definition of computable function.
\end{remark}

It follows that computable functions are continuous. Since we will
work with functions which are not necessarily continuous everywhere,
we shall consider functions which are computable on some subset of
$X$. More precisely, a function $T$ is said to be \defin{computable
on D} ($D\subset X$) if there is a uniform  sequence $(U^X_{n})_n$ of r.e open subsets of $X$ such it holds $T^{-1}(B_n)\cap D=U^X_{n}\cap D$  for the uniform sequence of ideal  balls $B_n$. $D$ is called the
\defin{domain of computability} of $T$.

\begin{remarks}

\item Since ideal balls generate the topology, a function is computable iff $T^{-1}(B^Y_n)$ is r.e open uniformly in $n$ (or the intersection of $D$ with a uniformly r.e open set).

\item If $T$ is computable then the images of ideal points can be uniformly computed, that is: $T(s_i^X)$ is a computable point, uniformly in $i$.

\item More generally, if $T$ is computable then there exists an algorithm which computes the image $T(x)$ of any $x$ in the following sense: the user enters some rational $\epsilon$ to the algorithm which, after asking finitely many times the user for finite approximations of $x$, halts outputting a finite approximation of $T(x)$ up to $\epsilon$.

\item The distance function $d:X\times X \to \R$ is a computable function.
\end{remarks}

%*********************************************
\subsection{Computable Probability Spaces (CPS)}\label{measures_section}

When $X$ is a computable metric space, the space of probability measures
over $X$, denoted by $\M(X)$, can be endowed with a structure of computable
metric space (this will be defined below, for more details, see \cite{Gac05, HoyRoj07}). Then a computable measure can be
defined as a computable point of $\M(X)$.

Some prerequisites from measure theory: We say that $\mu_n$ converge weakly to $\mu$ and write $\mu_n\to \mu$ if
\begin{equation}
\mu_n\to\mu \mbox{ iff } \mu_n f \to \mu f
\mbox{ for all real continuous bounded } f
\end{equation}
where $\mu f$ stands for $\int fd\mu$. Let us recall the Portmanteau
theorem. We say that a Borel set $A$ is
\defin{$\boldsymbol{\mu}$-continuous} if $\mu(\partial A)=0$, where
$\partial A=\overline{A}\cap\overline{X\setminus A}$ is the boundary of
$A$.

\begin{theorem}[Portmanteau theorem]\label{theorem_portmanteau}
Let $\mu_n,\mu$ be Borel probability measures on a separable metric space
$(X,d)$. The following are equivalent:
\begin{enumerate}
\item $\mu_n$ converges weakly to $\mu$,
\item $\limsup_n \mu_n(F)\leq \mu(F)$ for all closed sets $F$,
\item $\liminf_n \mu_n(G)\geq \mu(G)$ for all open sets $G$,
\item $\lim_n \mu_n(A)=\mu(A)$ for all $\mu$-continuity sets $A$.
\end{enumerate}
\end{theorem}

This theorem easily implies the following: when $(X,d)$ is a separable metric space, weak convergence can be proved using the following criterion:
\begin{proposition}\label{prop_portmanteau}
Let $\A$ be a countable basis of the topology which is closed under the
formation of finite unions. If $\mu_n(A)\to\mu(A)$ for every $A\in\A$, then
$\mu_n$ converge weakly to $\mu$.
\end{proposition}

%\begin{proof}
%Let $G$ be an open set: it can be expressed as a countable union of
%elements of $\A$: $G=\bigcup_i A_i$. Let $G_i=A_0\cup\ldots\cup A_i\in\A$:
%as $G_i \subseteq G$, $\mu_n(G_i)\leq \mu_n(G)$ so
%$\mu(G_i)=\lim_n \mu_n(G_i)\leq \liminf_n \mu_n(G)$. It follows that
%$\mu(G)=\sup_i\mu(G_i)\leq \liminf_n \mu_n(G)$, so condition (iii) of the
%preceding theorem holds. 
%\end{proof}

Let us introduce on $\M(X)$ the structure of a computable metric space. 
Let us endow $\M(X)$ with the weak topology, which is the topology of weak
convergence. As $X$ is separable and complete, so is $\mathcal{M}(X)$. Let
$D\subset \mathcal{M}(X)$ be the set of those probability measures that are
concentrated in finitely many points of $S$ and assign rational values to
them. It can be shown that this is a dense subset (\cite{Bil68}).

We consider the Prokhorov metric $\rho$ on $\mathcal{M}(X)$ defined by:
\begin{equation*}  \label{prokhorov}
\rho(\mu,\nu):=\inf \{\epsilon \in \mathbb{R}^+ :
\mu(A)\leq\nu(A^{\epsilon})+\epsilon \mbox{ for every Borel set }A\}.
\end{equation*}
where $A^{\epsilon}=\{x:d(x,A)< \epsilon\}$.

This metric induces the weak topology on $\mathcal{M}(X)$. Furthermore, it
can be shown that the triple $(\mathcal{M}(X),D,\rho)$ is a computable
metric space (see \cite{Gac05}, \cite{HoyRoj07}).

\begin{definition}
\label{compmeas}A measure $\mu $ is computable if there is an algorithmic
enumeration of a fast sequence of ideal measures $(\mu _{n})_{n\in \mathbb{N}%
}\subset D$ converging to $\mu $ in the Prokhorov metric and hence, in the
weak topology.
\end{definition}

The following theorem gives a characterization for the computability of measures in terms of the computability of the measure of sets (for a proof see \cite{HoyRoj07}):

\begin{theorem}\label{mu_computable}
A measure $\mu\in \M(X)$ is computable if and only if the measure $\mu(B_{i_1}\cup\ldots\cup B_{i_k})$ of finite unions of ideal open balls is lower-semi-computable uniformly in $\uple{i_1,\ldots,i_k}$.
\end{theorem}

\begin{definition}
A \defin{Computable Probability Space (CPS)} is a pair $(\X,\mu)$
where $\X$ is a computable metric space and $\mu$ is a computable
Borel probability measure on $X$.
\end{definition}

\begin{definition}Let $(\X,\mu)$ and $(\Y,\nu)$ be two \CPSs. A \defin{morphism}  from $(\X,\mu)$ to $(\Y,\nu)$ is a measure-preserving function  $F:X\to Y$ which is computable on a constructive $G_{\delta}$-set of $\mu$-measure one.
\end{definition}

We recall that $F$ is measure-preserving if $\nu(A)=\mu(F^{-1}(A))$ for every Borel set $A$. Computable probability structures can be easily
transferred:

\begin{proposition}\label{transfer}
Let $(\X,\mu)$ be a computable probability space, $\Y$ be a computable
metric space and $F:X\to Y$ a function which is computable on a
constructive $G_\delta$-set of $\mu$-measure one. Then the induced measure
$\mu_F$ on $Y$ defined by $\mu_F(A)=\mu(F^{-1}(A))$ is computable and $F$
is a morphism of computable probability spaces.
\end{proposition}

%***********************************************
\subsection{Algorithmic randomness}
Now we consider a generalization of Martin-L\"of tests to computable
probability spaces.
 Let $(\X,\mu)$ be a computable probability space.

\begin{definition}
A \defin{Martin-L\"of $\boldsymbol{\mu}$-Test} is a sequence $(U_n)_{n \in \N}$ of uniformly r.e open sets which satisfy $\mu(U_n)<2^{-n}$ for all $n$. Any subset of $\bigcap_nU_n$ is called an \defin{effective $\boldsymbol\mu$-null set}.
\end{definition}

\begin{definition}
A point $x\in X$ is called \defin{$\boldsymbol\mu$-random} if $x$ is
contained in no effective $\mu$-null set. The set of $\mu$-random
points is denoted $R_\mu$.
\end{definition}

Note that $\mu(R_\mu)=1$. The following is the generalization for metric spaces of a classical result in Cantor space due to Martin-L\"of. It says that the set of non-random points is not only a null set but an effective null set. For a proof see \cite{HoyRoj07}.

\begin{theorem}
The union of all effective $\mu$-null sets, denoted by $\mathcal{N}_{\mu}$, is again an effective $\mu$-null set.
\end{theorem}

Thus, there is a single Martin-L\"of test (often called \emph{universal}) which tests non-randomness, and $R_\mu=\mathcal{N}_{\mu}^c$.

We will need the following results, also taken from  \cite{HoyRoj07}.

\begin{lemma}\label{included}
Every $\mu$-random point is in every r.e open set of full measure.
\end{lemma}

\begin{proposition}[Morphisms of CPS preserve randomness]\label{morphism_random}

Let $F$ be a morphism of computable probability spaces $(\X,\mu)$
and $(\mathcal{Y},\nu)$. Then every $\mu$-random point $x$ is in the
domain of computability of $F$ and $F(x)$ is $\nu$-random.
\end{proposition}

%**********************************
\subsection{Kolmogorov complexity}
The idea is to define, for a finite object, the  minimal amount of
algorithmic information from which the object can be recovered. That is, the length of the shortest description (code) of the object. Since this shortest description is supposed to contain all necessary information to reconstruct in an algorithmic way the coded finite object, the Kolmogorov Complexity is also called \emph{Algorithmic Information Content}. For a complete introduction to Kolmogorov complexity we refer to a standard text \cite{Vit93}.

Let $\Sigma^*$ and $\Sigma^{\N}$ be the sets of finite and infinite words (over the finite alphabet $\Sigma$) respectively. A word $w \in \Sigma^*$ defines the \emph{cylinder} $[w]\subset \Sigma^{\N}$ of all possible continuations of $w$. A set $D=\{w_1,w_2,...\}\subset \Sigma^*$ defines an open set $[D]=\cup_i[w_i]\subset \Sigma^{\N}$. $D$ is called prefix-free if no word of $D$ is prefix of another one, that is if the cylinders $[w_i]$ are pairwise disjoint.

Let $X$ be $\Sigma^*$ or $\N$ or $\N^*$.

\begin{definition}
An \defin{interpreter} is a partial recursive function
$\FI:\{0,1\}^{*}\rightarrow X$ which has a prefix-free domain.
\end{definition}

\begin{definition}Let $I:\{0,1\}^{*}\rightarrow X$ be an interpreter. The \defin{complexity} (or \defin{Information Content}) $K_I(x)$ of $x \in X$ is defined to be
\[
K_I(x)=\left\{ \begin{array}{ll}
 |p| & \textrm{ if }p \textrm{ is a shortest input such that }I(p)=x \\
\infty & \textrm{ if there is no }p\textrm{  such that }I(p)=x
\end{array} \right.
\]
\end{definition}

It turns out that there exists an algorithmic enumeration of all interpreters, which entails the existence of a universal interpreter $U$ which is asymptotically optimal in the sense that the \emph{invariance theorem} holds:

\begin{theorem}[Invariance theorem] For all interpreter $I$ there exists $c_I \in \mathbb{N}$ such that for all $x \in X$ we have $K_U(w)\leq K_I(x)+c_I$.
\end{theorem}

We fix a universal interpreter $U$ and we let $K(x)=K_U(x)$.
\subsubsection{Estimates}\label{Kolmo_estimates}Let us recall
some simple estimates of complexity. Let $f,g$ be real-valued functions. We
say that $g$ \defin{additively dominates} $f$ and write $f\leqplus g$ if
there is a constant $c$ such that $f\leq g+c$. As codes are always
\emph{binary} words, we use base-2 logarithms, which we denote by
$\log$. We define $J(x)=x+2\log(x+1)$ for $x\geq 0$.

For $n\in\N$, $K(n)\leqplus J(\log n)$. For $n_1,\ldots,n_k\in \N$,
$K(n_1,\ldots,n_k)\leqplus K(n_1)+\ldots+K(n_k)$. The following property is
a version of a result attributed to Kolmogorov, stated in terms of prefix
complexity instead of plain complexity.

\begin{proposition}\label{prop_complexity_stratified_set}
Let $E\subseteq \N\times X$ be a r.e. set such that $E_n=\{x:(n,x)\in E\}$
is finite for all $n$. Then for $(s,n)$ with $s\in E_n$,
$$
K(s)\leqplus J(\log|E_n|)+K(n)
$$
\end{proposition}

\begin{proposition}\label{prop_complexity_measure}
Let $\mu$ be a computable measure on $\Sigma^\N$. For all $w\in\Sigma^*$,
$$
K(w)\leqplus -\log \mu([w])+K(|w|)
$$
\end{proposition}

\begin{theorem}[Coding theorem]\label{coding_theorem}
Let $P:X\to \Rplus$ be a lower semi-computable function such that $\sum_x P(x)\leq 1$. Then $K(x)\leqplus -\log P(x)$, i.e. there is a constant $c$ such that $K(x)\leq -\log P(x)+c$ for all $x\in X$.
\end{theorem}

Moreover, $\sum_x 2^{-K(x)}\leq 1$ as it is the Lebesgue measure of the
domain of the universal interpreter $U$. There is a relation between
Kolmogorov complexity and randomness, initial segments of random infinite
strings being maximally complex.

\begin{theorem}[Chaitin, Levin]\label{theorem_deficiency}Let $\mu$ be a
  computable measure. Then $\omega \in \Sigma^{\N}$ is a $\mu$-random
  sequence if and only if $\exists m$ $\forall n$ $K(\omega_{1:n})\geq-\log
  \mu[\omega_{1:n}]-m$.
\end{theorem}

The minimal such $m$, defined by $d_\mu(\omega):=\sup_n \{-\log \mu[\omega_{1:n}]-K(\omega_{1:n})\}$ and called the \defin{randomness deficiency} of $\omega$ w.r.t $\mu$, is not only finite almost everywhere: it has finite mean, that is $\int d_\mu(\omega) d\mu\leq 1$. For a proof see \cite{Vit93}.

%*******************************
\section{Effective symbolic dynamics and statistics of random points}\label{esd}

Let $(\X,\mu)$ be a computable probability space and let $R_\mu$ be the set of random points. The aim of this section is to study the set $R_\mu$ from a dynamical point of view. That is, we will put a dynamic $T$ on $(\X,\mu)$ (an endomorphism of computable probability space), and look at the abilities of random points (which are \emph{a priori} independent of $T$) to describe the statistical properties of $T$.

We recall that a Borel set $A$ is called \defin{$T$-invariant} if $T^{-1}(A)=A$(mod 0) and that the transformation $T$ is said to be \defin{ergodic} if every $T$-invariant set has measure 0 or 1.

%*******************************
\subsection{Symbolic dynamics of random points}\label{section_effective_symbolic}

Let $T$ be an endomorphism of the (Borel) probability space $(X,\mu)$. In the classical construction, one considers access to the system given by a finite measurable partition, that is a finite collection of pairwise disjoint Borel sets $\mathcal{P}=\{p_1,\ldots,p_k\}$ such that $\mu(\cup_i p_i)=1$. Then, to $(X,\mu,T)$ is associated a \emph{symbolic dynamical system} $(X_{\P},\sigma)$ (called the symbolic model of $(X,T,\P)$). The set $X_{\P}$ is a subset of $\{1,2,\ldots,k\}^{\mathbb{N}}$. To a point $x\in X$ corresponds an infinite sequence $\omega=(\omega_i)_{i\in \mathbb{N}}=\phi_{\P}(x)$ defined by:
\[
\phi_{\P}(x)=\omega \Leftrightarrow \forall j \in \mathbb{N}, T^j(x)\in p_{\omega_j}
\]
The transformation $\sigma : X_{\P} \rightarrow X_{\P}$ is \emph{the shift} defined by $\sigma((\omega_i)_{i\in \mathbb{N}})=(\omega_{i+1})_{i \in \mathbb{N}}$.

As $\P$ is a measurable partition, the map $\phi_{\P}$ is measurable and then the measure $\mu$ induces the measure $\mu_{\P}$ (on the associated symbolic model) defined by $\mu_{\P}(B)=\mu(\phi_{\P}^{-1}(B))$ for all measurable $B \subset X_{\P}$.

The requirement of $\phi_{\P}$ being measurable makes the symbolic
model appropriate from the measure-theoretic view point, but is not
enough to have a symbolic model compatible with the computational
approach:

\begin{definition}
Let $T$ be an endomorphism of the computable probability space
$(\X,\mu)$ and $\P=\{p_1\dots,p_k\}$ a finite measurable
partition. The associated symbolic model $(X_{\P},\mu_{\P},\sigma)$ is
said to be \textit{\textbf{an effective symbolic model}} if the map
$\phi_{\P}:X \to \{1,\dots,k\}^{\N}$ is a morphism of CPS (here the space $\{1,\dots,k\}^{\N}$ is endowed with the standard computable structure).
\end{definition}

The sets $p_i$ are called the \textit{\textbf{atoms}} of
$\mathcal{P}$ and we denote by $\mathcal{P}(x)$ the atom containing
$x$ (if there is one).
Observe that $\phi_\P$ is computable on its domain only if the atoms are open r.e sets (in the
domain):

\begin{definition}[Computable partitions]\label{mu-partitions}
A measurable partition $\mathcal{P}$ is said to be a
\textit{\textbf{computable partition}} if its atoms are r.e open
sets.
\end{definition}

Conversely:

\begin{theorem}\label{symbolic}
Let $T$ be an endomorphism of the CPS $(X,\mu)$ and $\P=\{p_1\dots,p_k\}$ a finite computable partition. Then the associated symbolic model is effective.
\end{theorem}

\begin{proof}
Let $D$ be the domain of computability of $T$ (it is a full-measure
constructive $G_\delta$). Define the set
$$
X^\P=D\cap \bigcap_{n\in\N} T^{-n}(p_1 \cup \ldots \cup p_k)
$$ $X^\P$ is a full-measure constructive $G_\delta$-set: indeed, as
$p_1\cup\ldots\cup p_k$ is r.e. and $T$ is computable on $D$ there are
uniformly r.e. open sets $U_n$ such that $D\cap T^{-n}(p_1\cup\ldots\cup
p_k)=D\cap U_n$, so $X^\P=D\cap\bigcap_n U_n$. As $T$ is
measure-preserving, all $U_n$ have measure one.

Now, $X^\P\cap \phi_\P^{-1}[i_0,\ldots,i_n]=X^\P\cap p_{i_0}\cap
T^{-1}p_{i_1}\cap\ldots\cap T^{-n}p_{i_n}$. This proves that $\phi_{\P}$ is computable over $X^\P$. Proposition
\ref{transfer} allows to conclude.
\end{proof}

After the definition an important question is: are there computable partitions? the answer depends on the existence of open r.e sets with a zero-measure boundary.

\begin{definition}\label{almost_decidable}
A set $A$ is said to be \defin{almost decidable} if there are two r.e open sets $U$ and $V$ such that:
$$
U\subset A, \quad V\subseteq \comp{A}, \quad \mu(U)+\mu(V)=1
$$
\end{definition}

\begin{remarks}
\item a set is almost decidable if and only if its complement is almost decidable,
\item an almost decidable set is always a continuity set,
\item a $\mu$-continuity ideal ball is always almost decidable,
\item unless the space is disconnected (i.e. has
non-trivial clopen subsets), no set can be \emph{decidable}, i.e.
semi-decidable (r.e) and with a semi-decidable complement (such a set must
be clopen\footnote{In Cantor space for example (which is totally
disconnected), every cylinder (ball) is a decidable set. Indeed, deciding
if some infinite sequence belongs to some cylinder reduces to
a finite pattern-matching.}). Instead, a set can be decidable \emph{with probability $1$}: there is an algorithm which decides if a point belongs to the set or not, for almost every point. That is why we call it \emph{almost decidable}.
\end{remarks}

Ignoring computability, the existence of open $\mu$-continuity sets directly follows from the fact that the collection of open sets is uncountable  and $\mu$ is finite. The problem in the computable setting is that there are only countable many open r.e sets. Fortunately, there still always exists a basis of almost decidable balls. This result, first obtained in \cite{HoyRoj07} with other techniques, will be used many times in the sequel, in particular it directly implies the existence of computable partitions. For completeness we present a different, self-contained proof.

\begin{theorem}
There is a family of uniformly computable reals $(r^i_n)_{i,n\in \N}$ such
that for all $i$, $\{r^i_n:n\in\N\}$ is dense in $\R^+$ and such that for
every $i,n$, the ball $B(s_i,r^i_n)$ is almost decidable.
\end{theorem}

\begin{proof}
Let $s_i$ be an ideal point. Put $I_{\uple{j,k}}=[q_j,q_k]$ with $q_j, q_k$
positive rational numbers. We show that for every $n=\uple{j,k}$ we can
compute, uniformly in $n$, a real $r^i_n\in I_n$ for which
$\mu(\partial B(s_i,r^i_n))=0$. First observe that for a closed interval
$I=[a,b]$ ($a,b\in \Q$), the complement of $B_I=\overline{B}(s_i,b)/B(s_i,a)$, is r.e open. Then by corollary \ref{mu_computable}, its measure
is lower semi-computable and then we can semi-decide for a given
rational $q$ the relation $\mu(B_I)<q$. The algorithm computing $r^i_n$
enumerates a sequence of nested closed intervals $(J_k)_{k\in\N}$
whose length tends to $0$, with $J_0=I_n$, and such that for all $k$,
$\mu(B_{J_k})<2^{-k+1}$. Then $\{r^i_n\}=\cap_{k\geq 1}J_k$. It works as follows:

In stage $k+1$ (the interval $J_k=[a,b]$ has already been found), put
$m=\frac{b-a}{3}$ and test in parallel $\mu(B_{[a,a+m]})<2^{-k}$ and
$\mu(B_{[b-m,b]})<2^{-k}$. Since $\mu(B_{J_k})<2^{-k+1}$, one of the
tests must stop, and then provides the ``good'' interval $J_{k+1}$ for
which the condition holds.
\end{proof}

We denote by $B^{\uple{i,n}}$ the almost decidable ball $B(s_i,r^i_n)$.

The family $\{B^\uple{i,n}:i,n\in\N\}$ is a basis for the topology. It is
even effectively equivalent to the basis of ideal balls : every ideal ball
can be expressed as a r.e. union of almost decidable balls, and vice-versa.

We finish presenting some results that will be needed in the next subsection.

\begin{corollary}\label{corollary_generating}
On every computable probability space, there exists a family of uniformly
computable partitions which generates the Borel $\sigma$-algebra.
\end{corollary}

\begin{proof}
Take $\P_\uple{i,n}=\{B(s_i,r^i_n),X\setminus \overline{B}(s_i,r^i_n)\}$ where $\overline{B}$ is the closed ball: as the almost decidable balls form a
basis of the topology, the $\sigma$-algebra generated by the $P_k$ is the
Borel $\sigma$-field.
\end{proof}

%Thus, in order to construct an effective symbolic model we have to be able to
%\emph{compute} the symbolic orbit of a point $x$. To do this, we
%must be able to \emph{decide} the atom $\P(x)$ to which $x$ belongs, as in the proof of the proposition above.

\begin{proposition}\label{decidable_set_measure}
If $A$ is almost decidable then $\mu(A)$ is a computable real number.
\end{proposition}

\begin{proof}
 Since $U$ and $V$ are r.e open, by theorem \ref{mu_computable} their measures
 are lower-semi-computable. As $\mu(U)+\mu(V)=1$, their measures are
 also upper-semi-computable.
\end{proof}

The following regards the computability of inducing a measure in a subset and will be used in the proof of prop. \ref{poinc}

\begin{proposition}\label{induced_measure}
Let $\mu$ be a computable measure and $A$ an almost decidable subset of
$X$. Then the induced measure $\mu_A(.)=\mu(.|A)$ is
computable. Furthermore, $R_{\mu_A}=R_\mu \cap A$.
\end{proposition}

\begin{proof}
let $W=B_{n_1}\cup\ldots\cup B_{n_k}$ be a finite union of ideal balls. $\mu_A(W)=\mu(W\cap A)/\mu(A)=\mu(W \cap U)/\mu(A)$. $W\cap U$ is a r.e open set, so its measure is lower semi-computable. As $\mu(A)$ is computable, $\mu_A(W)$ is lower semi-computable. Note that everything is uniform in $\uple{n_1,\ldots,n_k}$. The result follows from theorem \ref{mu_computable}.

Let $U$ and $V$ as in the definition of an almost decidable set. First note
that $R_\mu\cap A=R_\mu\cap U$, as $R_\mu\subseteq U\cup V$ by lemma
\ref{included}. Again by lemma \ref{included}, $R_{\mu_A}\subseteq U$, and
as $\mu_A\leq \frac{1}{\mu(A)}\mu$, every $\mu$-effective null set is also
a $\mu_A$-effective null set, so $R_{\mu_A}\subseteq R_\mu$. Hence, we have
$R_{\mu_A}\subseteq R_\mu\cap U$.

Conversely, $\comp{R}_{\mu_A}$ being a $\mu_A$-effective null set, its
intersection with $U$ is a $\mu$-effective null set, by definition of
$\mu_A$. So $\comp{R}_{\mu_A}\cap U\subseteq \comp{R}_\mu$, which is
equivalent to $R_\mu\cap U \subseteq R_{\mu_A}$.
\end{proof}

%*****************************************

\subsection{Some statistical properties of random points}\label{secbir}

With the tools developed so far, it is possible to translate many results of the form

%\begin{align*}
$$
 \mu\{x: P(x)\}=1,
$$
%\end{align*}

with $P$ some predicate, into an ``individual'' result of the  form:

%\begin{align*}
$$
\mbox{``If }x\mbox{ is }\mu\mbox{-random, then }P(x)\mbox{''}.
$$
%\end{align*}

In this section we give two examples: recurrence and statistical typicality.

\begin{definition}Let $X$ be a metric space. A point $x\in X$ is said to be \defin{recurrent} for a transformation $T:X\to X$, if $\liminf_n d(x,T^n x)=0$.
\end{definition}
%Cristo: definition of recurrence

\begin{proposition}[Random points are recurrent]\label{poinc}
Let $(X,\mu)$ be a computable probability space. If $x$ is $\mu$-random, then it is recurrent with respect to {\em every } measure preserving endomorphism $T$ on $(X,\mu)$.
%it must be ergodic, isn't it? no
\end{proposition}

\begin{proof}
take $x\in R_\mu$ and $B$ an almost decidable neighborhood of $x$. Then
$\mu(B)>0$ and there is a r.e open set $U$ such that:
$$
\bigcup_{n\geq 1}T^{-n}B = U\cap D
$$
where $D$ is the domain of computability of $T$. By the Poincar\'e
recurrence theorem, this set has full measure for $\mu_B(.)=\mu(.|B)$. By
proposition \ref{induced_measure}, $x\in R_{\mu_B}$, so by lemma
\ref{included}, $x$ is in $U$.
\end{proof}

We now prove that random points satisfy a stronger property to be used in the sequel: statistical typicality. Let us then introduce this concept.

Let $X$ be a metric space and $T$ be a continuous transformation on $X$. Let $C_b(X)$ be the space of bounded real-valued continuous functions on $X$. For $f\in C_b(X)$ define:

\begin{equation}\label{limit}
\overline{f}(x):=\lim_{n \to \infty} \frac{1}{n} \sum_{j=0}^{n-1}f(T^jx)
\end{equation}

at the points $x$ where this limit exists. We recall that a point $x$ is
called \defin{generic} for $T$ if $\overline{f}(x)$ is defined for every
$f\in C_b(X)$.

Every generic point $x$ generates a probability measure $\mu_x$ which is
invariant for $T$, dually defined by:

\begin{equation}
\int_X f d\mu_x = \overline{f} (x) \mbox{ for all $f\in C_b(X)$}.
\end{equation}

In other words, $x$ is generic if the measure
$\nu_n=\frac{1}{n}\sum_{j<n}\delta_{T^j x}$ converges weakly to
$\mu_x$, where $\delta_y$ is the Dirac probability measure concentrated on $y$. Let $\mu$ be an ergodic measure for $T$. A generic point $x$ is
said to be \defin{$\boldsymbol{\mu}$-typical} if $\mu_x=\mu$. The
well-known Birkhoff ergodic theorem says that for each ergodic measure
$\mu$, the set of $\mu$-typical points has $\mu$-measure one.

From a statistical point of view, $\mu$-typical points are those whose orbits  reproduce the main statistical features of $\mu$ (in particular they are a total measure set), hence in some sense they are \emph{random} for the dynamic.

\bigskip

What algorithmically random points have to do with dynamically random points?

\bigskip

This problem has already been studied by V'yugin
(\cite{Vyu97}) in the particular case of the Cantor space and for computable
observables. We prove a general version which applies to computable
dynamics on any computable probability space, for any bounded continuous (not necessarily computable) observable. The strategy is simple: we
use computable partitions to construct effective symbolic models and use the following particular case of V'yugin's main theorem.

\begin{lemma}\label{vyugin} Let $\mu$ be a computable shift-invariant ergodic measure on the Cantor space $\{0,1\}^\omega$. Then for each $\mu$-random sequence $\omega$:
\begin{equation}
\lim_n \frac{1}{n} \sum_{i=0}^n \omega_i = \mu([1])
\end{equation}
\end{lemma}

We are now able to prove:

\begin{theorem}\label{birkhoff_theorem}
Let $(X,\mu)$ be a computable probability space. Then each $\mu$-random point $x$ is $\mu$-typical for {\em every} ergodic endomorphism $T$.
\end{theorem}

We remark that the theorem holds uniformly for all bounded continuous observables and all ergodic endomorphisms.

\begin{proof}  Let $f_A$ be the characteristic function of the set $A$. First, let us show that if $A$ is an almost decidable set
 then for all $\mu$-random point $x$:

\begin{equation}\label{birk}
\lim_n \frac{1}{n} \sum_{i=0}^n f_A\circ T^i(x) = \mu(A)
\end{equation}

Indeed, consider the \emph{computable partition} defined by $\P:=\{U,V\}$
with $U$ and $V$ as in definition \ref{almost_decidable} and the associated
symbolic model $(X_{\P},\sigma,\mu_{\P})$. By proposition \ref{symbolic},
$\phi_{\P}(x)$ is a well defined $\mu_{\P}$-random infinite sequence, so
lemma \ref{vyugin} applies and gives (\ref{birk}). This can be reformulated
as the convergence of $\nu_n(A)$  to $\mu(A)$. Now, the collection
of almost decidable sets satisfies proposition \ref{prop_portmanteau}, so
$\nu_n$ converges weakly to $\mu$: $x$ is $\mu$-typical.
\end{proof}

\input{entropy}

%\section{Complexity as a measure theoretical and topological invariant}

%$\Kshad$ should be a topological invariant if $X$ is compact.

%Moreover, we should have $\sup_x\Kshad(x,T)=h_{top}(T)$.

%...And it is clearly a measure topological invariant.

\bibliographystyle{alpha}
\bibliography{bibliography}{}

%\begin{corollary}If $X$ is compact and $T$ computable then it holds:
%$$
%\sup \{\Ksym(x,T):x \mbox{ random w.r.t some invariant }\mu\}=h_{top}(T).
%$$
%\end{corollary}

%\input{biblio}

\end{document}

%% file: macros_english.tex
\usepackage[colorlinks=true]{hyperref}
\usepackage[latin1]{inputenc}
\usepackage{mathrsfs}
\usepackage{latexsym}

\newcommand{\uple}[1]{{\langle #1 \rangle}}

\newcommand{\FI}{\varphi}

\newcommand{\A}{\mathcal{A}}

\newcommand{\G}{\mathcal{G}}

\newcommand{\M}{\mathcal{M}}
\newcommand{\N}{\mathbb{N}}
\renewcommand{\O}{\mathcal{O}}
\renewcommand{\P}{\mathcal{P}}
\newcommand{\Q}{\mathbb{Q}}
\newcommand{\R}{\mathbb{R}}
\renewcommand{\S}{\mathcal{S}}

\newcommand{\X}{\mathcal{X}}

\newcommand{\comp}[1]{{#1}^{\textrm{c}}}

\newcommand{\defin}[1]{\emph{\textbf{#1}}}
\newcommand{\puce}{$\bullet$ }

\newcommand{\leqplus}{<\!\!\!\!\!^{\!^+}~}%\overset{_+}{<}}%\stackrel{+}{\leq}}

\newcommand{\Rplus}{\overline{\R}^{_+}}

\newcommand{\diam}{\mathrm{diam}}

\newcommand{\eff}{\mathrm{eff}}
\newcommand{\lK}{\underline{\K}}
\newcommand{\uK}{\overline{\K}}
\newcommand{\K}{\mathcal{K}}
\newcommand{\Ksym}{\K_\mu}

\newcommand{\CPSs}{computable probability spaces}

\newcommand{\Y}{\mathcal{Y}}

\newtheorem{lemma}{Lemma}[subsection]
\newtheorem{proposition}{Proposition}[subsection]
\newtheorem{theorem}{Theorem}[subsection]
\newtheorem*{theorem_star}{Theorem}
\newtheorem{corollary}{Corollary}[subsection]

\theoremstyle{definition}
\newtheorem{definition}{Definition}[subsection]

\theoremstyle{remark}
\newtheorem{remark}{Remark}[subsection]

\newtheorem*{Examples}{Examples}

\newenvironment{examples}{\noindent\begin{Examples}\begin{enumerate}}{\end{enumerate}\end{Examples}}
\newenvironment{remarks}{\noindent\textbf{remarks:}\begin{itemize}}{\end{itemize}}

\makeindex

%% file: entropy.tex
%*****************************************
%*****************************************
\section{Measure-theoretic entropies}\label{section_measure_entropy}
Suppose discrete objects (symbolic strings for instance) are produced by
some source. The tendency of the source toward producing such object more
than such other can be modeled by a probability distribution, which gives
more information than the crude set of possible outcomes. The Shannon
entropy of the probabilistic source measures the degree of uncertainty
that lasts when taking the probability distribution into account.

Any ergodic dynamical system $(X,T,\mu)$ can be seen as a source of
outputs. Kolmogorov and Sina\"i adapted Shannon's theory to dynamical
systems in order to measure the degree of unpredictability or chaoticity of an
ergodic system. The first step consists in discretizing the space $X$ using
finite partitions. Let $\xi=\{C_1,\ldots,C_n\}$ be a finite measurable
partition of $X$. Then let $T^{-1}\xi$ be the partition whose atoms are the
pre-images $T^{-1}C_i$. Then let
\[
\xi_n=\xi \vee T^{-1}\xi \vee T^{-2}\xi \vee \ldots \vee T^{-(n-1)}\xi
\]
be the partition given by the sets of the form
\[
C_{i_0}\cap T^{-1}C_{i_1}\cap \ldots \cap T^{-(n-1)}C_{i_{n-1}},
\]
varying $C_{i_j}$ among all the atoms of $\xi$. Knowing which atom $\xi_n$
a point $x$ belongs to comes to knowing which atoms of the partition
$\xi$ the orbit of $x$ visits up to time $n-1$.

The measure-theoretical entropy of the system w.r.t the partition $\xi$ can
then be thought as the rate (per time unit) of gained information (or
removed uncertainty) when observations of the type ``$T^{n}(x)\in C_{i}$''
are performed. This is of great importance when classifying dynamical
systems: it is a measure-theoretical invariant, which enables one to
distinguish non-isomorphic systems.

We briefly recall the definition. For more details, we refer the reader to
\cite{Bil65}, \cite{Wal82}, \cite{Pet83}, \cite{HK95}.

%$$
%\xymatrix{
%& Local & Global \\
%Shannon & I_\mu(x,T)\ar@{=}[d]^{random}\ar@{=}[r]^{\mu-a.e.}_{Sh-McM-Br} & h_\mu(x,T) \\
%Kolmogorov & \Ksym(x,T) & 
%}
%$$
%*****************************************
\subsection{Entropy with Shannon information}\label{entropy}
Given a partition $\xi$ and a point $x$, $\xi(x)$ denotes the atom of the
partition $x$ belongs to. Let us consider the \defin{Shannon information
function} relative to the partition $\xi_n$ (the information which is
gained by observing that $x\in \xi_n(x)$),
$$
I_\mu(x|\xi_n):=-\log \mu( \xi_n(x))$$
and its mean, the entropy of the partition $\xi_n$,
$$
H_\mu(\xi_n):=\int_X I_\mu(.|\xi_n)d\mu = \sum_{C\in\xi_n}-\mu(C)\log \mu(C)
$$

The \defin{measure-theoretical} or \defin{Kolmogorov-Sina\"i entropy} of
$T$ relative to the partition $\xi$ is defined as:
\[
h_{\mu}(T,\xi)=\lim_{n \rightarrow \infty} \frac{1}{n} H_{\mu}(\xi_n).
\]
(which exists and is an infimum, since the sequence $H_{\mu}(\xi_n)_n$ is
sub-additive). With the Shannon information function, it is possible to define a
kind of point-wise notion of entropy with respect to a partition
$\xi$:
$$
\limsup_n \frac{1}{n}I_{\mu}(x|\xi_n).
$$
This local entropy is related to the global entropy of the system by the
celebrated Shannon-McMillan-Breiman theorem:
\begin{theorem_star}[Shannon-McMillan-Breiman]
Let $T$ be an ergodic endomorphism of the probability space
$(X,\mathscr{B},\mu)$ and $\xi$ a finite measurable partition. Then for
$\mu$-almost every $x$,
\begin{equation}\label{shannon_mcmillan}
\lim_{n\to\infty}\frac{1}{n}I_\mu(x|\xi_n)=h_\mu(T,\xi).
\end{equation}
The convergence also holds in $L^1(X,\mathscr{B},\mu)$.
\end{theorem_star}

Now we suppress the partition-dependency: the \defin{Kolmogorov-Sina\"i entropy} of $(X,T,\mu)$ is
\[
h_{\mu}(T):=\sup\{h_{\mu}(T,\xi):\xi \text{  finite measurable partition}\}
\]

We recall the following two results that we will need
later. The first proposition follows directly from the definitions.

\begin{proposition}\label{sym}
If $(\Sigma^\N,\mu_{\xi},\sigma)$ is the symbolic model associated to $(X,\mu,T,\xi)$ then $h_{\mu}(T,\xi)=h_{\mu_{\xi}}(\sigma)$.
\end{proposition}

The next proposition is taken from \cite{Pet83}:

\begin{proposition}\label{petersen}
If $(\xi_i)_{i\in\N}$ is a family  of finite measurable partitions which
generates the Borel $\sigma$-field up to sets of measure 0, then
$h_{\mu}(T)=\sup_ih_{\mu}(T,\xi_0\vee...\vee\xi_i)$.
\end{proposition}

%********************************
\subsection{Entropy with Kolmogorov information}

In this section, $T$ is an endomorphism of the computable probability space
$(X,\mu)$ and $\xi=\{C_1,\ldots,C_k\}$ is a computable partition. Let
$(\Sigma^\N,\mu_\xi,\sigma)$ be the effective symbolic model of
$(X,\mu,T,\xi)$ where $\Sigma=\{1,\ldots,k\}$ (see section \ref{section_effective_symbolic}).

Kolmogorov introduced his algorithmic information content (also called Kolmogorov complexity) as a quantity of information,
on the same level as Shannon information. When the measure, the
transformation and the partition are computable, it makes sense to define
the algorithmic equivalents of the notions defined above. It turns out that
the two points of view are strongly related.

An atom $C$ of the partition $\xi_n$ can then be seen as a word of length
$n$ on the alphabet $\Sigma$, which allows one to consider its Kolmogorov
complexity $K(C)$. For those points whose all iterates are covered by $\xi$
(they form a constructive dense $G_\delta$ of full measure), we define the
\defin{Kolmogorov information function} relative to the partition $\xi_n$:
$$
\mathcal{I}(x|\xi_n):= K(\xi_n(x))
$$
which is independent of $\mu$. We can then define \defin{algorithmic entropy} of the
partition $\xi_n$ as the mean of $\mathcal{I}$:
$$
\mathcal{H}_\mu(\xi_n):=\int_X \mathcal{I}(.|\xi_n)d\mu = \sum_{C\in\xi_n}\mu(C)K(C).
$$

We also define a local notion of algorithmic entropy, which we call
symbolic orbit complexity:
\begin{definition}[Symbolic orbit complexity]\label{ksym}
Let $T$ be an endomorphism of the computable probability space
$(X,\mu)$. For any finite computable partition $\xi$, we define
$\Ksym(x,T|\xi) := \limsup_n\frac{1}{n}\mathcal{I}(x|\xi_n)$. Then, we can suppress the dependence on $\xi$ by taking the supremum over all computable partitions:
$$
\Ksym(x,T) := \sup\{\Ksym(x,T|\xi):\xi\textrm{ computable partition}\}
$$
\end{definition}

As there are only countably many computable partitions, $\Ksym(x,T)$ is
defined almost everywhere (at least on random points).

The quantity $\Ksym(x,T|\xi)$ was introduced by Brudno in
\cite{Bru83} without any computability restriction on the space, the
measure nor the transformation. He proved:
\begin{theorem}[Brudno]\label{theorem_brudno1}
$\Ksym(x,T|\xi)=h_\mu(T,\xi)$ for $\mu$-almost every point.
\end{theorem}

\begin{remark}
Already Brudno remarked that if $x$ has not an eventually periodic orbit, by taking the supremum of $\Ksym(x,T|\xi)$ over all -- not necessarily computable -- finite partitions 
$\xi$ generally gives an infinite quantity, that is why Brudno did not go further 
(he did not have a computable structure at his disposal), and proposed a topological
 definition using open covers instead of partitions. 
\end{remark}

 We will compare $\Ksym$ and Brudno orbit complexity in  section \ref{oc}, we now show that the hypothesis of definition \ref{ksym} enables one 
 to derive Brudno's theorem in a rather simple manner.

%********************************

The theory of randomness and Kolmogorov complexity on the space of symbolic
sequences provides powerful results (theorem \ref{theorem_deficiency} and
proposition \ref{prop_complexity_measure}) which enable to relate the
algorithmic entropies $\mathcal{I}_\mu$ and $\mathcal{H}_\mu$ to the
Shannon entropies $I_\mu$ and $H_\mu$ (inequalities
(\ref{information_functions}), (\ref{entropies})). We recall these two results:
if $\Sigma^\N$ is endowed with a computable probability measure $\nu$, then
for all $\omega\in \Sigma^\N$,
\begin{equation}\label{complexity_measure}
\begin{array}{ccccc}
-\log\nu[\omega_{0..n-1}]-d_{\nu}(\omega) & \leq & K(\omega_{0..n-1}) &
 \leqplus & -\log\nu[\omega_{0..n-1}]+K(n)
\end{array}
\end{equation}
where $d_{\nu}$ is the deficiency of randomness, which satisfies
$\int_{\Sigma^\N} d_{\nu}d\nu<1$ and is finite exactly on Martin-L\"of
random sequences (the constant in $\leqplus$ does not depend on $\omega$
and $n$, see section \ref{Kolmo_estimates}).

\subsection{Equivalence between local entropies}
Applying (\ref{complexity_measure}) to $\nu=\mu_\xi$ directly gives:
\begin{equation}\label{information_functions}
\begin{array}{ccccc}
I_\mu(.|\xi_n)-d_\mu\circ\phi_\xi & \leq & \mathcal{I}(.|\xi_n) & \leqplus
& I_\mu(.|\xi_n)+K(n)
\end{array}
\end{equation}
where it is defined (almost everywhere, at least on random
points). Every $\mu$-Martin-L\"of random point $x$ is mapped by $\phi_\xi$ on a
$\mu_\xi$-Martin-L\"of random sequence (see proposition \ref{morphism_random}), whose
randomness deficiency is finite. It then follows that the local entropies
using Shannon information and Kolmogorov information coincide on
$\mu$-random points:
\begin{proposition}
\begin{equation}\label{local_entropies_random_points}
\Ksym(x,T|\xi)=\limsup_n\frac{1}{n}I(x|\xi_n)\quad\text{for every
  $\mu$-Martin-L\"of random point $x$}
\end{equation}
\end{proposition}
This equality together with the Shannon-McMillan-Breiman theorem (\ref{shannon_mcmillan}) give
directly Brudno's theorem (theorem \ref{theorem_brudno1}).

\begin{remark}[Equivalence between global entropies]
Now, the Kolmogorov-Sina\"i entropy, originally expressed using Shannon
entropy, can be expressed using algorithmic entropy. Taking the mean
in (\ref{information_functions}), one obtains:
\begin{equation}\label{entropies}
\begin{array}{ccccc}
H_\mu(\xi_n)-1 & \leq & \mathcal{H}_\mu(\xi_n) & \leqplus & H_\mu(\xi_n)+K(n)
\end{array}
\end{equation}

So,
$$
h_\mu(T|\xi)=\lim_n
\frac{H_\mu(\xi_n)}{n}=\lim_n\frac{\mathcal{H}_\mu(\xi_n)}{n}
$$
As the collection of computable partitions is generating (see corollary
\ref{corollary_generating}), the Kolmogorov-Sina\"i entropy of $(X,\mu,T)$
can be characterized by:
\[
h_{\mu}(T)=\sup\left\{\lim_n\frac{\mathcal{H}_\mu(\xi_n)}{n}:\xi \textrm{  finite computable partition}\right\}.
\]

It then follows that $\Ksym(x,T)=h_\mu(T)$ for $\mu$-almost every $x$. We
now strengthen this, proving that it holds for all Martin-L\"of
random points.
\end{remark}

%********************************
\subsection{Orbit complexity vs entropy}

On the Cantor space, V'yugin (\cite{Vyu97}) and later Nakamura
(\cite{Nak05}) proved a slightly weaker version of the
Shannon-McMillan-Breiman for Martin-L\"of random sequences. In particular,
we will use:

\begin{theorem}[V'yugin]\label{cantor_Shannon_random}
Let $\mu$ be a computable shift-invariant ergodic measure on
$\Sigma^{\N}$. Then, for any $\mu$-Martin-L\"of random sequence $\omega$,
$$
 \limsup_{n\rightarrow \infty} -\frac{1}{n}\log \mu([\omega_{0..n-1}])=h_{\mu}(\sigma).
$$
\end{theorem}

Note that it is not known yet if the limit exists for all random sequences.

Using effective symbolic models, this can be easily extended to any
computable probability space.

\begin{corollary}[Shannon-McMillan-Breiman for random points]\label{Shannon_random}
Let $T$ be an ergodic endomorphism of the computable probability space
$(X,\mu)$, and $\xi$ a computable partition. For every $\mu$-Martin-L\"of random point $x$,
$$
 \limsup_{n\rightarrow \infty} -\frac{1}{n}\log \mu(\xi_n(x))=h_{\mu}(T,\xi).
$$
\end{corollary}

\begin{proof}
Since $\xi$ is computable, the symbolic model $(X_{\xi},\mu_{\xi},\sigma)$
is effective. Every $\mu$-Martin-L\"of random point $x$ is mapped to a
$\mu_\xi$-Martin-L\"of random sequence $\omega$, for which the
preceding theorem holds. Using the facts that
$\mu(\xi_n(x))=\mu_\xi([\omega_{0..n-1}])$ and
$h_\mu(T,\xi)=h_{\mu_\xi}(\sigma)$ allows to conclude.
\end{proof}

Finally, this implies our first announced result:

\begin{theorem}\label{final}
Let $T$ be an ergodic endomorphism of the computable probability space
$(X,\mu)$. For every $\mu$-Martin-L\"of random point $x$:
$$
\Ksym(x,T)=h_{\mu}(T).
$$
\end{theorem}

\begin{proof}
We combine equality (\ref{local_entropies_random_points}) and corollary
\ref{Shannon_random}: for every random point $x$, $\Ksym(x,T|\xi)=\limsup_n
\frac{1}{n}I_\mu(x|\xi_n)=h_\mu(T,\xi)$. Since the collection of all
computable partitions generates the Borel $\sigma$-field (corollary
\ref{corollary_generating}), $\Ksym(x,T)=\sup\{h_\mu(T,\xi):\xi\mbox{ computable
partition}\}=h_\mu(T)$ (proposition \ref{petersen}).
\end{proof}

%**********************

\subsection{Orbit complexity} \label{oc}

In this section, $(X,d,\S)$ is a computable metric space and $T:X\to X$ a
transformation (for the moment, no continuity or computability assumption
is put on $T$). We will consider a notion of orbit complexity
 which quantifies the algorithmic
information needed to describe the orbit of $x$ with finite but arbitrarily
accurate precision. The definition we will give coincide with Brudno's original definition on
compact spaces (see \cite{Gal00}).

Given $\epsilon>0$ and $n\in\N$, the
algorithmic information that is needed to list a sequence of ideal points which follows the orbit of $x$ for $n$ steps at a distance less than $\epsilon $ is:
$$
\K_n(x,T,\epsilon) := \min\{K(i_0,\ldots,i_{n-1}):
d(s_{i_j},T^jx)<\epsilon \text{ for }j=0,\ldots,n-1\}
$$
where $K$ is the self-delimiting Kolmogorov complexity.

We then define the
maximal and minimal growth-rates of this quantity:
\begin{eqnarray*}
\uK(x,T,\epsilon) & := &
\limsup_{n\to\infty}\frac{1}{n}\K_n(x,T,\epsilon) \\
\lK(x,T,\epsilon) & := &
\liminf_{n\to\infty}\frac{1}{n}\K_n(x,T,\epsilon).
\end{eqnarray*}

As $\epsilon$ tends to $0$, these quantities increase (or at least do not
decrease), hence they have limits (which can be infinite).

\begin{definition}
The \defin{upper} and \defin{lower orbit complexities} of $x$ under $T$ are defined by:
\begin{eqnarray*}
\uK(x,T) & := & \lim_{\epsilon \to 0^+}\uK(x,T,\epsilon) \\
\lK(x,T) & := & \lim_{\epsilon \to 0^+}\lK(x,T,\epsilon).
\end{eqnarray*}
\end{definition}

\begin{remark} If $T$ is computable, and assuming that $\epsilon$ takes
  only rational values, the $n$ first iterates of $x$ could be
$\epsilon$-shadowed by the orbit of a single ideal point instead of a
pseudo-orbit of $n$ ideal points. Actually it is easy to see that it gives
the same quantities $\uK(x,T,\epsilon)$ and $\lK(x,T,\epsilon)$: let $\K'_n(x,T,\epsilon)=\min\{K(i):
d(T^js_i,T^jx)<\epsilon \text{ for } j<n\}$, one has:
\begin{eqnarray*}
\K'_n(x,T,2\epsilon) & \leqplus & \K_n(x,T,\epsilon) + K(\epsilon) \\
\K_n(x,T,\epsilon) & \leqplus & \K'_n(x,T,\epsilon/2) + K(n,\epsilon)
\end{eqnarray*}
Indeed, from $\epsilon$ and $i_0,\ldots,i_{n-1}$ some ideal point can be
algorithmically found in the constructive open set
$B(s_{i_0},\epsilon)\cap\ldots\cap T^{-(n-1)}B(s_{i_{n-1}},\epsilon)$,
uniformly in $i_0,\ldots,i_{n-1}$. Its $n$ first iterates
$2\epsilon$-shadow the orbit of $x$, which proves the first inequality. For
the second inequality, some $i_0,\ldots,i_{n-1}$ can be algorithmically
found from $n$, $\epsilon$, and a point $s_i$ whose $n$ first iterates $\epsilon/2$-shadow the orbit of $x$,
taking any $s_{i_j}\in B(T^js_i,\epsilon/2)$.
\end{remark}

\begin{remark}
Under the same assumptions, one could define $K(B_n(s_i,\epsilon))$ to be
  $K(i,n,\epsilon)$, and replace $K(i)$ by $K(B_n(s_i,\epsilon))$ in the
  definition of $\K'_n(x,T,\epsilon)$, without changing the quantities
  $\uK(x,T,\epsilon)$ and $\lK(x,T,\epsilon)$. Indeed,
$$
K(i)\leqplus K(B_n(s_i,\epsilon)) \leqplus K(i)+K(n)+K(\epsilon)
$$
\end{remark}

%*****************************
%*****************************
\section{Equivalence of the two notions of orbit complexity for random points}\label{section_equivalence}

We now prove:
\begin{theorem}\label{shadsym}Let $T$ be an ergodic endomorphism of the computable probability space $(X,\mu)$, where $X$
  is compact. Then for every Martin-L\"of random point $x$,
$$
\uK(x,T)=\Ksym(x,T).
$$
\end{theorem}

\begin{proof}[Proof of $\uK(x,T)\leq \Ksym(x,T)$]
Let $\epsilon>0$. Choose a computable partition $\xi$ of diameter
$<\epsilon$ (this is why we require $X$ to be compact). To every cell of
$\xi$, associate an ideal point which is inside (as $\xi$ is computable,
this can be done in a computable way, but we actually do not need
that). The translation of symbolic sequences in sequences of ideal points
through this finite dictionary is constructive, and transforms the symbolic
orbit of a point $x$ into a sequence of ideal points which is
$\epsilon$-close to the orbit of $x$. So $\uK(x,T,\epsilon)\leq
\Ksym(x,T|\xi)$. The inequality follows letting $\epsilon$ tend to $0$.
\end{proof}

To prove the other inequality, we recall some technical stuff. The
self-delimiting Kolmogorov complexity of natural numbers $k\geq 1$ satisfies
$$
K(k)\leqplus f(k)
$$ where $f(x)=\log x+1+2\log(\log x +1)$ for all $x\in \R,x \geq 1$. $f$
is a concave increasing function and $x\mapsto xf(1/x)$ is an increasing
function on $]0,1]$ which tends to $0$ as $x\to 0$.

We recall that for finite sequences of natural numbers $(k_1,\ldots,k_n)$, one
has
\begin{align*}
K(k_1,\ldots,k_n)\leqplus K(k_1)+\ldots+K(k_n)
\end{align*}
as the shortest descriptions for $k_1,\ldots,k_n$ can be extracted from
their concatenation (this is one reason to use the self-delimiting
complexity instead of the plain complexity).

\begin{lemma}
\label{discr_lemma}
Let $\Sigma$ be a finite alphabet and $n\in\N$. Let $u,v \in \Sigma^n$ and $0<\alpha<1/2$
such that the density of the set of positions where $u$ and $v$ differ is less
than $\alpha$, that is:
$$
\frac{1}{n}\#\{i\leq n: u_i\neq v_i\} < \alpha < 1/2
$$

Then $\left|\frac{1}{n}K(u)-\frac{1}{n}K(v)\right| < \alpha
f(1/\alpha)+\frac{c}{n}$ where $c$ is a constant independent of $u,v$ and $n$.
\end{lemma}

\begin{proof}
Let $(i_1,\ldots,i_p)$ be the ordered sequence of indices where $u$ and $v$
differ. By hypothesis, $p/n<\alpha$. Put $k_1=i_1$ and $k_j=i_j-i_{j-1}$ for
$2\leq j\leq p$.

We now show that $u$ can be recovered from $v$ and roughly $\alpha
f(1/\alpha) n$ bits more. Indeed $u$ can be computed from
$(v,k_1,\ldots,k_p)$, constructing the string which
coincides with $v$ everywhere but at positions
$k_1,k_1+k_2,\ldots,k_1+\ldots+k_p$, so $K(u)\leqplus
K(v)+K(k_1)+\ldots+K(k_p)\leqplus K(v)+f(k_1)+\ldots+f(k_p)$.

Now, as $f$ is a concave increasing function, one has:
$$
\frac{1}{p}\sum_{j\leq p}f(k_j) \leq
f\left(\frac{1}{p}\sum_{j\leq p}k_j\right) = f\left(\frac{i_p}{p}\right) \leq f\left(\frac{n}{p}\right)
$$

As a result,
$$
\frac{1}{n}K(u) \leq \frac{1}{n}K(v) +
\frac{p}{n}f\left(\frac{n}{p}\right) + \frac{c}{n}
$$ where $c$ is some constant independent of $u,v,n,p$. As $p/n<\alpha<1/2$
and $x\mapsto x f(1/x)$ is increasing for $x\leq 1/2$, one has:
$$
\frac{1}{n}K(u) \leq \frac{1}{n}K(v) + \alpha f(1/\alpha) + \frac{c}{n}
$$

Switching $u$ and $v$ gives the result ($c$ may be changed).
\end{proof}

We are now able to prove the other inequality.

\begin{proof}[Proof of $\Ksym(x,T)\leq \uK(x,T)$]
Fix some computable partition $\xi$. We show that for any $\beta>0$ there is
some $\epsilon>0$ such that for every Martin-L\"of random point $x$,
$\Ksym(x,T|\xi)\leq \uK(x,T,\epsilon)+\beta$. As $\uK(x,T,\epsilon)$
increases as $\epsilon\to 0^+$ and $\beta$ is arbitrary, the inequality
follows.

First take $\alpha<1/2$ such that $\alpha f(1/\alpha)<\beta$, and remark that
$$
\lim_{\epsilon\to 0^+} \mu\left(\overline{(\partial
  \xi)^\epsilon}\right) = \mu(\partial \xi) = 0
$$

Hence there is some $\epsilon$ such that $\mu\left(\overline{(\partial
\xi)^{2\epsilon}}\right)<\alpha$. From a sequence of ideal points we will
reconstruct the symbolic orbit of a random point with a density of errors
less than $\alpha$. Lemma \ref{discr_lemma} will then allow to conclude.

We define an algorithm $\A(\epsilon,i_0,\ldots,i_{n-1})$ with $\epsilon\in
\Q_{>0}$ and $i_0,\ldots,i_{n-1} \in\N$ which outputs a word $a_0\ldots
a_{n-1}$ on the alphabet $\xi$. To compute $a_j$, $\A$ semi-decides in a
dovetail picture:
\begin{itemize}
\item $s_{i_j}\in C$ for every $C\in\xi$,
\item $s \in C$ for every $s\in B(s_{i_j},\epsilon)$ and every $C\in\xi$.
\end{itemize}

The first test which stops provides some $C\in\xi$: put $a_j=C$.

Let $x$ be a random point whose iterates are covered by $\xi$, and
$s_{i_0},\ldots,s_{i_{n-1}}$ be ideal points which $\epsilon$-shadow the
first $n$ iterates of $x$. We claim that $\A$ will halt on
$(\epsilon,i_0,\ldots,i_{n-1})$. Indeed, as $T^jx$ belongs to some
$C\in\xi$, $C\cap B(s_{i_j},\epsilon)$ is a non-empty open set and then
contains at least one ideal point $s$, which will be eventually dealt with.

We now compare the symbolic orbit of $x$ with the symbolic sequence
computed by $\A$. A discrepancy at rank $j$ can appear only if $T^jx\in
(\partial \xi)^{2\epsilon}$. Indeed, if $T^jx\notin (\partial
\xi)^{2\epsilon}$ then $B(T^jx,2\epsilon)\subseteq C$ where $C$ is the cell
$T^jx$ belongs to. As $d(s_{i_j},T^jx)<\epsilon$,
$B(s_{i_j},\epsilon)\subseteq B(x,2\epsilon)\subseteq C$, so the algorithm
gives the right cell.

Now, as $x$ is typical,
$$
\limsup_{n\to\infty} \frac{1}{n}\#\{j<n: T^jx\in (\partial
  \xi)^{2\epsilon}\}\leq \mu\left(\overline{(\partial
  \xi)^{2\epsilon}}\right)<\alpha
$$ so there is some $n_0$ such that for all $n\geq n_0$,
$\frac{1}{n}\#\{j<n: T^jx\in (\partial \xi)^{2\epsilon}\}<\alpha$. This
implies that for all $n\geq n_0$ and ideal points
$s_{i_0},\ldots,s_{i_{n-1}}$ which $\epsilon$-shadow the first $n$ iterates
of $x$ and with minimal complexity, the algorithm
$\A(\epsilon,i_0,\ldots,i_{n-1})$ produces a symbolic string $u$ which
differs from the symbolic orbit $v$ of $x$ of length $n$ with a density of
errors $<\alpha$. As $K(u)\leqplus K(\epsilon)+\K_n(x,T,\epsilon)$ and
$\alpha f(1/\alpha)<\beta$, applying lemma \ref{discr_lemma} gives:
\begin{eqnarray*}
\frac{1}{n}K(\xi_n(x)) = \frac{1}{n}K(v) & \leq & \frac{1}{n}K(u) +
\alpha f(1/\alpha) + \frac{c}{n} \\ & \leq &
\frac{1}{n}\left(\K_n(x,T,\epsilon)+K(\epsilon)+c'\right) + \beta + \frac{c}{n}
\end{eqnarray*}
where $c'$ is independent of $n$. Taking the $\limsup$ as $n\to\infty$
gives:
$$
\Ksym(x,T|\xi) \leq \uK(x,T,\epsilon) + \beta
$$
\end{proof}

Combining theorems \ref{shadsym} and \ref{final}, we obtain a version of
Brudno's theorem (theorem \ref{brudno_theorem}) for Martin-L\"of random
points.

\begin{corollary}
Let $T$ be an ergodic endomorphism of the computable probability space
  $(X,\mu)$, where $X$ is compact. Then for every Martin-L\"of
  random point $x$:
$$
\uK(x,T)=h_\mu(T)
$$
\end{corollary}

%*****************************************
%*****************************************
\section{Topological entropies}\label{section_topological_entropies}

%Let $X$ be a compact metric space and $T:X\to X$ a continuous map. The
%topological entropy of the system $(X,T)$ measures the growth rate of the
%number of distinguishable orbits of the system. It is a topological
%invariant, in particular it does not depend on the metric inducing the
%topology. The original definition by Adler, Konheim and McAndrew
%(\cite{AKM65}) uses open covers of the space. Brudno's definition
%of algorithmic complexity of orbits is underlay by this definition, and
%then is defined only for compact spaces.
%
%Bowen (\cite{Bow71}) gave a characterization of the topological entropy of
%a system inspired of the $\epsilon$-entropy of Kolmogorov and Tikhomirov
%(\cite{KolTik59}). Using this idea, in \cite{Gal00} the notion of orbit complexity was extended 
%to the non-compact case. We denote by $\uK(x,T)$ the algorithmic
% complexity of the orbit of $x$ under
%$T$.
%
%Using topological and algorithmic arguments, in this section we prove:
%\begin{theorem_star}
%Let $X$ be a compact computable metric space and $T:X\to X$ a computable
%map. Then,
%$$
%h(T)=\sup_{x\in X}\uK(x,T).
%$$
%\end{theorem_star}
%where $h(T)$ is the topological entropy of the system.

%%******************************
%\subsection{Topological entropy}
Bowen's definition of topological entropy is reminiscent of the capacity (or box counting
dimension) of a totally bounded subset of a metric space. In order to find relations with orbit complexity we will also use another characterization of
topological entropy, expressing it as a kind of Hausdorff dimension. We
first present Bowen's definition.

In this section, $X$ is a metric space and $T:X\to X$ a continuous map.

%*****************************************
\subsection{Entropy as a capacity}
We recall the definition: for $n\geq 0$, let us define the distance
$d_n(x,y)=\max\{d(T^ix,T^iy):0\leq i<n\}$ and the Bowen ball
$B_n(x,\epsilon)=\{y:d_n(x,y)<\epsilon\}$, which is open by continuity of
$T$. Given a totally bounded set $Y\subseteq X$ and numbers $n\geq
0,\epsilon>0$, let $N(Y,n,\epsilon)$ be the minimal cardinality of a cover
of $Y$ by Bowen balls $B_n(x,\epsilon)$. A set of points $E$ such that
$\{B_n(x,\epsilon):x\in E\}$ is a cover of $Y$ is also called an
$(n,\epsilon)$-spanning set of $Y$. One then defines:

$$
h_1(T,Y,\epsilon)=\limsup_{n\to \infty}\frac{\log N(Y,n,\epsilon)}{n}
$$
which is non-decreasing as $\epsilon\to 0$, so the following limit exists:
$$
h_1(T,Y)=\lim_{\epsilon\to 0}h_1(Y,T,\epsilon).
$$
When $X$ is compact, the \defin{topological entropy} of $T$ is $h(T)=h_1(T,X)$. It measures the exponential growth-rate of the number of distinguishable orbits of the system.

\begin{remark}
The topological entropy can be defined using separated sets instead of open
covers: a subset $A$ of $X$ is $(n,\epsilon)$-separated if for any distinct points
$x,y\in A$, $d_n(x,y)>\epsilon$. Let us define $M(Y,n,\epsilon)$ as the maximal
cardinality of an $(n,\epsilon)$-separated subset of $Y$. It is easy to see
that $M(Y,n,2\epsilon)\leq N(Y,n,\epsilon)\leq M(Y,n,\epsilon)$, and
hence $h_1(T,Y)$ can be alternatively defined using $M(Y,n,\epsilon)$ in place of $N(Y,n,\epsilon)$.
\end{remark}

%*****************************************
\subsection{Entropy as a dimension}
It is possible to define a topological entropy which is an analog of Hausdorff
dimension. His definition coincides with the classical one in the
compact case. Hausdorff dimension has stronger stability properties than
box dimension, which has important consequences, as we will see in what
follows. We refer the reader to \cite{Pes98}, \cite{HK02} for more details.

Let $X$ be a metric space and $T:X\to X$ a continuous map. The
$\epsilon$-size of $E\subseteq X$ is $2^{-s}$ where 
$$
s=\sup\{n\geq 0: \diam(T^iE)\leq\epsilon \mbox{ for }0\leq i<n\}.
$$
It measures how long the orbits starting from $E$ are $\epsilon$-close. As $\epsilon$ decreases, the $\epsilon$-size of $E$ is non-decreasing. The $2\epsilon$-size of a Bowen ball $B_n(x,\epsilon)$ is less than $2^{-n}$.

In a way that is reminiscent from the definition of Hausdorff measure, let
us define
$$
m^s_\delta(Y,\epsilon)=\inf_\G\left\{\sum_{U\in\G}(\epsilon\text{-size}(U))^s\right\}
$$ where the infimum is taken over all countable covers $\G$ of $Y$ by open
sets of $\epsilon$-size $<\delta$. This quantity is monotonically
increasing as $\delta$ tends to $0$, so the limit
$m^s(Y,\epsilon):=\lim_{\delta\to 0^+}m^s_\delta(Y,\epsilon)$ exists and is
a supremum. There is a critical value $s_0$ such that
$m^s(Y,\epsilon)=\infty$ for $s<s_0$ and $m^s(Y,\epsilon)=0$ for
$s>s_0$. Let us define $h_2(T,Y,\epsilon)$ as this critical value:
\[
h_2(T,Y,\epsilon) := \inf\left\{s: m^s(Y,\epsilon)=0\right\} =  \sup\left\{s:m^s(Y,\epsilon)=\infty\right\}.
\]
As less and less covers are allowed when $\epsilon\to 0$ (the $\epsilon$-size of sets does not decrease), the following limit exists
\[
h_2(T,Y) := \underset{\epsilon\to 0^+}{\lim} h_2(T,Y,\epsilon)
\]
and is a supremum. In \cite{Pes98}, it is proved that:

%Equivalently, we could define $m^s(Y)=\lim_{\epsilon\to 0}m^s(Y,\epsilon)=\sup_{\epsilon>0}m^s(Y,\epsilon)$ and $h_2(T,Y)=\inf\{s:m^s(Y)=0\}=\sup\{s:m^s(Y)=\infty\}$. Indeed,
%\begin{eqnarray*}
%s\geq \inf\{s:m^s(Y)=0\} & \iff & \forall s'>s, m^{s'}(Y)=0 \\
%& \iff & \forall s'>s,\epsilon>0, m^{s'}(Y,\epsilon)=0 \\
%& \iff & \forall \epsilon>0,s'>s, m^{s'}(Y,\epsilon)=0 \\
%& \iff & \forall \epsilon>0, s\geq h_2(T,Y,\epsilon) \\
%& \iff & s\geq h_2(T,Y)
%\end{eqnarray*}

\begin{theorem}\label{theorem_pesin}
When $Y$ is a $T$-invariant compact set, $h_1(T,Y)=h_2(T,Y)$.% One even has: $h_1(T,Y,\epsilon)=h_2(T,Y,\epsilon)$.
\end{theorem}

In particular, if the space $X$ is compact, then $h(T)=h_1(T,X)=h_2(T,X)$.
%******************************
%***************************
%**************************
\subsection{Orbit complexity vs entropy}
Now we prove the main theorem of the section:

\begin{theorem}[Topological entropy vs orbit complexity]\label{theorem_top_entropy}
Let $X$ be a compact computable metric space, and $T:X\to X$ a computable map. Then
$$
h(T)=\sup_{x\in X}\lK(x,T)=\sup_{x\in X}\uK(x,T).
$$
\end{theorem}

In order to prove this theorem, we define an effective version of the
topological entropy, which is strongly related to the complexity of
orbits.

%****************************
\subsubsection{Effective entropy as an effective dimension}
Before defining an effective version, we give a simple characterization
which will accommodate to effectivisation.

\begin{definition}
A \defin{null $\boldsymbol{s}$-cover} of $Y\subseteq X$ is a set
$E\subseteq \N^3$ such that:
\begin{enumerate}
\item $\sum_{(i,n,p)\in E} 2^{-sn}<\infty$,
\item for each $k,p\in\N$, the set $\{B_n(s_i,2^{-p}): (i,n,p)\in E,n\geq k\}$ is a cover of $Y$.
\end{enumerate}
\end{definition}

The idea is simple: every null $s$-cover induces open covers of arbitrary
small size and arbitrary small weight. Remark that any null $s$-cover of
$Y$ is also a null $s'$-cover for all $s'>s$.

\begin{lemma}\label{cover_lemma}
$h_2(T,Y)=\inf\{s:Y\mbox{ has a null $s$-cover}\}$.
\end{lemma}

\begin{proof}
Suppose $s>h_2(T,Y)$. We fix $p,k\in\N$ and put $\epsilon=2^{-p}$ and
$\delta=2^{-k}$. As $m^s_{\delta}(Y,\epsilon)=0$, there is a cover
$(U_{j,k,p})_j$ of $Y$ by open sets of $\epsilon$-size $\delta_{j,k,p}<\delta$ with
$\sum_j \delta_{j,k,p}^s<2^{-(k+p)}$. Let $s_i$ be any ideal point in
$U_{j,k,p}$. If $\delta_{j,k,p}>0$, then $\delta_{j,k,p}=2^{-n}$ for some
$n$. If $\delta_{j,k,p}=0$, take any $n\geq (j+k+p)/s$. In both
cases, $U_{j,k,p}$ is included in the Bowen ball $B_n(s_i,\epsilon)$. We
define $E_{k,p}$ as the set of $(i,n,p)$ obtained this way, and
$E=\bigcup_{k,p}E_{k,p}$. By construction,
for each $k,p$, $\{B_n(s_i,2^{-p}):(i,n,p)\in E,n\geq k\}$ is a cover of $Y$. Moreover, $\sum_{(i,n,p)\in E_{k,p}}2^{-sn}\leq \sum_j
\delta_{j,k,p}^s +\sum_j 2^{-(j+k+p)}\leq 2^{-(k+p)+2}$, so $\sum_{(i,n,p)\in E}2^{-sn}<\infty$.

Conversely, if $Y$ has a null $s$-cover $E$, take $\epsilon,\delta>0$ and $p,k$ such
that $\epsilon>2^{-p+1}$ and $\delta>2^{-k}$. For all $k'\geq k$, the family
$\{B_n(s_i,2^{-p}):(i,n,p)\in E,n\geq k'\}$ is a cover of $Y$ by open
sets of $\epsilon$-size smaller than $2^{-n}\leq\delta$. Moreover,
$\sum_{(i,n,p)\in E,n\geq k'} 2^{-sn}$ tends to $0$ as $k'$ grows, so
$m^s_\delta(Y,\epsilon)=0$. It follows that $s\geq h_2(T,Y)$.
\end{proof}

By an \emph{effective} null $s$-cover, we mean a null $s$-cover $E$ which
is a r.e. subset of $\N^3$.

\begin{definition}
The \defin{effective topological entropy} of $T$ on $Y$ is defined by
$$
h^\eff_2(T,Y)=\inf\{s:Y\mbox{ has an effective null $s$-cover}\}
$$
\end{definition}

As less null $s$-covers are allowed in the effective version, $h_2(T,Y)\leq
h^\eff_2(T,Y)$. Of course, if $Y\subseteq Y'$ then $h^\eff_2(T,Y)\leq
h^\eff_2(T,Y')$. We now prove:

\begin{theorem}[Effective topological entropy vs lower orbit complexity]\label{theorem_eff_top_ent_orbit_complexity}
Let $X$ be an effective metric space and $T:X\to X$ a continuous map. For all $Y\subseteq X$,
$$
h^\eff_2(T,Y)=\sup_{x\in Y} \lK(x,T)
$$
\end{theorem}
which implies in particular that $h^\eff_2(T,\{x\})=\lK(x,T)$: the
restriction of the system to a single orbit may have \emph{positive}
effective topological entropy.

This kind of result has already been obtained for the Hausdorff dimension
of subsets of the Cantor space, proving that the effective dimension of a
set $A$ is the supremum of the lower growth-rate of Kolmogorov complexity
of sequences in $A$ (which corresponds to theorem
\ref{theorem_eff_top_ent_orbit_complexity} for sub-shifts). This remarkable
property is a counterpart of the countable stability property of Hausdorff
dimension ($\dim Y=\sup_i \dim Y_i$ when $\bigcup_i Y_i=Y$) (see
\cite{CaiHar94}, \cite{May01}, \cite{Lut03}, \cite{Rei04}, \cite{Sta05}).

Theorem \ref{theorem_eff_top_ent_orbit_complexity} is a direct consequence of the two following lemmas.

\begin{lemma}\label{h_less_k}
Let $\alpha\geq 0$ and $Y_\alpha=\{x:\lK(x,T)\leq\alpha\}$. One has $h^\eff_2(T,Y_\alpha)\leq \alpha$.
\end{lemma}

\begin{proof}
Let $\beta>\alpha$ be a rational number. We define the r.e. set
$E=\{(i,n,p):K(i,n,p)<\beta n\}$. Let $p\in\N$ and $\epsilon=2^{-p}$. If
$x\in Y_\alpha$ then $\lK(x,T,\epsilon)\leq
\alpha<\beta$ so for infinitely many $n$, there is some
$s_i$ such that $x\in B_n(s_i,\epsilon)$ and $K(i,n,p)< \beta n$. So for
all $k$, $\{B_n(s_i,2^{-p}):(i,n,p)\in E,n\geq k\}$ covers
$Y_\alpha$. Moreover, $\sum_{(i,n,p)\in E} 2^{-\beta n}\leq
\sum_{(i,n,p)\in E}2^{-K(i,n,p)}\leq 1$.

$E$ is then an effective null $\beta$-cover of $Y_\alpha$, so
$h^\eff_2(T,Y_\alpha)\leq \beta$. And this is true for every rational $\beta>\alpha$.
\end{proof}

\begin{lemma}\label{k_less_h}
Let $Y\subseteq X$. For all $x\in Y$, $\lK(x,T)\leq h^\eff_2(T,Y)$.
\end{lemma}

\begin{proof}
Let $s>h^\eff_2(T,Y)$: $Y$ has an effective null $s$-cover $E$. As
$\sum_{(i,n,p)\in E} 2^{-s n}<\infty$, by the coding theorem $K(i,n,p)\leq
sn+c$ for some constant $x$, which does not depend on $i,n,p$. If $x\in Y$,
then for each $p,k$, $x$ is in a ball $B_n(s_i,2^{-p})$ for some $n\geq k$
with $(i,n,p)\in E$. Then $\K_n(x,T,2^{-p})\leq sn+c$ for infinitely many $n$, so
$\lK(x,T,2^{-p})\leq s$. As this is true for all $p$,
$\lK(x,T)\leq s$. As this is true for all $s>h^\eff_2(T,Y)$, we can conclude.
\end{proof}

\begin{proof}[Proof of theorem \ref{theorem_eff_top_ent_orbit_complexity}]
By lemma \ref{k_less_h}, $\alpha:=\sup_{x\in Y} \lK(x,T)\leq
h^\eff_2(T,Y)$. Now, as $Y\subseteq Y_\alpha$, $h^\eff_2(T,Y)\leq
h^\eff_2(T,Y_\alpha)\leq \alpha$ by lemma \ref{h_less_k}.
\end{proof}

The definition of an effective null $\alpha$-cover involves a summable
computable sequence. The universality of the sequence
$2^{-K(i)}$ among summable lower semi-computable
sequences is at the core of the proof of the preceding theorem, which
states that there is a universal effective null $\alpha$-cover, for every
$\alpha\geq 0$. In other words, there is a maximal set of effective
topological entropy $\leq\alpha$, and this set is $Y_\alpha=\{x\in
X:\lK(x,T)\leq\alpha\}$.

The definition of the topological entropy as a capacity could be also made
effective, restricting to effective covers. Classical capacity does not
share with Hausdorff dimension the countable stability. For the same
reason, its effective version is not related with the orbit complexity as
strongly as the effective topological entropy is. Nevertheless, a weaker
relation holds, which is sufficient for our purpose: the upper complexity
of orbits is bounded by the effective capacity. We do not develop this and
only state the needed property (which implicitly uses the fact that the
effective capacity coincides with the classical capacity for a compact
computable metric space):

\begin{lemma}\label{lemma_upper_complexity_entropy}
Let $X$ be a compact computable metric space. For all $x\in X$,
$\uK(x,T)\leq h_1(T,X)$.
\end{lemma}

\begin{proof}
We first construct a r.e. set $E\subseteq \N^3$ such that for each $n$,
$p$, $\{s_i:(i,n,p)\in E\}$ is a $(n,2^{-p})$-spanning set and a
$(n,2^{-p-2})$-separated set. Let us fix $n$ and $p$ and enumerate
$E_{n,p}=\{i:(i,n,p)\in E\}$, in a uniform way. The algorithm starts with
$S=\emptyset$ and $i=0$. At step $i$ it analyzes $s_i$ and decides to add
it to $S$ or not, and goes to step $i+1$. $E_{n,p}$ is the set of points
which are eventually added to $S$.

\begin{description}
\item[Step $\boldsymbol{i}$] for each ideal point $s\in
S$, test in parallel $d_n(s_i,s)<2^{-p-1}$ and $d_n(s_i,s)>2^{-p-2}$: at
least one of them must stop. If
the first one stops first, reject $s_i$ and go to Step $i+1$. If the second
one stops first, go on with the other points $s\in S$: if all $S$ has been
considered, then add $s_i$ to $S$ and go to Step $i+1$.
\end{description}

By construction, the set of selected ideal points forms a
$(n,2^{-p-2})$-separated set. If there is $x\in X$ which is at distance at
least $2^{-p}$ from every selected point, then let $s_i$ be an ideal point
$s_i$ with $d_n(x,s_i)<2^{-p-1}$: $s_i$ is at distance at least $2^{-p-1}$
from every selected point, so at step $i$ it must have been
selected, as the first test could not stop. This is a contradiction: the
selected points form a $(n,2^{-p})$-spanning set.

From the properties of $E_{n,p}$ it follows that $N(X,n,2^{-p})\leq |
E_{n,p}|\leq M(X,n,2^{-p-2})$, and then 
$$
\sup_p\left(\limsup \frac{1}{n}\log | E_{n,p}|\right)=h_1(T,X)
$$
If $\beta>h_1(T,X)$ is a rational number, then for each $p$, there is
$k\in \N$ such that $\log | E_{n,p}|<\beta n$ for all $n\geq k$.

Now, for $s_i\in E_{n,p}$, $K(i)\leqplus \log |E_{n,p}|+2\log\log
|E_{n,p}|+K(n,p)$ by proposition \ref{prop_complexity_stratified_set}. Take $x\in X$: $x$ is in some $B_n(s_i,2^{-p})$ for each
$n$, so $\uK(x,T,2^{-p})\leq \limsup_n \frac{1}{n}\log |E_{n,p}|\leq \beta$
as $\log|E_{n,p}|<\beta n$ for all $n\geq k$. As this is true for all $p$
and all $\beta>h_1(T,X)$, $\uK(x,T)\leq h_1(T,X)$ and this for all $x\in
X$.
\end{proof}

We are now able to prove 
theorem \ref{theorem_top_entropy}. Combining the several results
established above:

%$$
%h(T)=h_2(T,X) \leq h^\eff_2(T,X)=\sup_{x\in X}\lK(x,T)\leq \sup_{x\in
%X}\uK(x,T)\leq h_1(T,X)=h(T)
%$$

%\[
%\begin{array}{rcll}
%h(T) & = & h_2(T,X) & \text{(theorem \ref{theorem_pesin})} \\
%h_2(T,X) & \leq & h^\eff_2(T,X) & \\
%h^\eff_2(T,X) & = & \sup_{x\in X}\lK(x,T) & \text{(theorem \ref{theorem_eff_top_ent_orbit_complexity})} \\
%\sup_{x\in X}\lK(x,T) & \leq & \sup_{x\in X}\uK(x,T) & \\
%\sup_{x\in X}\uK(x,T) & \leq & h_1(T,X) & \text{(lemma \ref{lemma_upper_complexity_entropy})} \\
%h_1(T,X) & = & h(T) & \text{(definition of $h(T)$)}
%\end{array}
%\]

\[
\begin{array}{ccccc}
h_1(T,X) = h_2(T,X) & \leq & h^\eff_2(T,X) = \sup_{x\in X}\lK(x,T) &
\leq & \sup_{x\in X}\uK(x,T)\leq h_1(T,X) \\
\text{(theorem \ref{theorem_pesin})} & &
\text{(theorem \ref{theorem_eff_top_ent_orbit_complexity})} & &
\text{(lemma \ref{lemma_upper_complexity_entropy})}
\end{array}
\]

and the statement is proved.